\documentclass[a4paper,12pt]{amsart}

\usepackage{amsmath,amsfonts,amsthm,amssymb}
\usepackage{fullpage}   
\usepackage[all]{xy}

\theoremstyle{plain}    
\newtheorem{theorem}{Theorem}[section]
\newtheorem{lemma}[theorem]{Lemma}
\newtheorem{proposition}[theorem]{Proposition}
\newtheorem{corollary}[theorem]{Corollary}

\theoremstyle{definition}   
\newtheorem{definition}[theorem]{Definition}
\newtheorem{example}[theorem]{Example}
\newtheorem{remark}[theorem]{Remark}
\newtheorem{construction}[theorem]{Nullification Construction}
\newtheorem*{definition*}{Definition}
\newtheorem*{remark*}{Remark}

\theoremstyle{remark}   

\newcommand{\Int}{\mathbb{Z}}   

\newcommand{\cc}{\mathcal{C}}   
\newcommand{\ab}{\mathcal{A}}   
\newcommand{\bb}{\mathcal{B}}   
\newcommand{\Derived}{\mathcal{D}}      
\newcommand{\tc}{\mathcal{T}}   
\newcommand{\ff}{\mathcal{F}}   

\newcommand{\Hom}{\mathrm{Hom}}   
\newcommand{\ext}{\mathrm{Ext}}   
\newcommand{\End}{\mathrm{End}}   
\newcommand{\im}{\mathrm{im}}   
\newcommand{\chains}{\mathrm{C}}    
\newcommand{\colim}{\operatornamewithlimits{colim}} 
\newcommand{\ee}{\mathcal{E}}   
\newcommand{\cell}{\mathrm{Cell}}   
\newcommand{\Null}{\mathrm{Null}}   

\newcommand{\simples}{\mathbb{S}} 

\newcommand{\aideal}{\mathfrak{a}}

\newcommand{\problem}[1]{ }
\newcommand{\elaboration}[1]{}
\newcommand{\discard}[1]{}      

\begin{document}

\title{Colocalization Functors in Derived Categories and Torsion Theories}
\author{Shoham Shamir}
\address{Department of Mathematics, University of Bergen, 5008 Bergen, Norway}
\email{shoham.shamir@math.uib.no}
\date{\today}

\begin{abstract}
Let $R$ be a ring and let $\ab$ be a hereditary torsion class of $R$-modules. The inclusion of the localizing subcategory generated by $\ab$ into the derived category of $R$ has a right adjoint, denoted $\cell_\ab$. In~\cite{Benson}, Benson shows how to compute $\cell_\ab R$ when $R$ is a group ring of a finite group over a prime field and $\ab$ is the hereditary torsion class generated by a simple module. We generalize Benson's construction to the case where $\ab$ is any hereditary torsion class on $R$. It is shown that for every $R$-module $M$ there exists an injective $R$-module $E$ such that:
\[ H^n(\cell_\ab M) \cong \ext_{\End_R(E)}^{n-1}(\Hom_R(M,E),E) \text{ for } n\geq 2\]
\end{abstract}

\maketitle

\section{Introduction}
\label{sec: Introduction}

Let $R$ be a ring and let $\Derived_R$ be the (unbounded) derived category of chain complexes of left $R$-modules. Fix a class $\ab$ of objects of $\Derived_R$. We recall some definitions of Dwyer and Greenlees from~\cite{DwyerGreenlees}. An object $N$ of $\Derived_R$ is \emph{$\ab$-null} if $\ext^*_R(A,N)=0$ for every $A\in \ab$. An object $C$ of $\Derived_R$ is \emph{$\ab$-cellular} if $\ext^*_R(C,N)=0$ for every $\ab$-null $N$. An $\ab$-cellular object $C$ is an \emph{$\ab$-cellular approximation} of $X\in \Derived_R$ if there is a map $\mu:C \to X$ such that $\ext^*_R(A,\mu)$ is an isomorphism for all $A\in\ab$. Finally, an $\ab$-null object $N$ is an \emph{$\ab$-nullification} of $X$ if there is a map $\nu:X \to N$ which is universal among maps in $\Derived_R$ from $X$ to $\ab$-null objects. Denote an $\ab$-cellular approximation of $X$ by $\cell_\ab X$ and an $A$-nullification of $X$ by $\Null_\ab X$.

The following properties are easy to check. A map $\mu:C \to X$ is an $\ab$-cellular approximation of $X$ if and only if it is universal among all maps from $\ab$-cellular objects to $X$. There is an exact triangle $\cell_\ab X \to X \to \Null_\ab X$ whenever $\cell_\ab X$ or $\Null_\ab X$ exists. An $\ab$-cellular approximation of some object $X$ is unique up to isomorphism and the same goes for an $\ab$-nullification of $X$.

Now suppose $\ab$ is a set, then it turns out that the full subcategory of $\ab$-cellular objects is the localizing subcategory generated by $\ab$ (see~\cite{DwyerGreenleesIyengar} and \cite[5.1.5]{Hirschhorn}). Moreover, when $\ab$ is a set the inclusion functor of the full subcategory of $A$-cellular objects into $\Derived_R$ has a right adjoint, which is $\cell_\ab X$ for every $X\in \Derived_R$, see~\cite{Hirschhorn} or~\cite{FarjounBook}. Hence $\cell_\ab$ can be constructed as a \emph{colocalization} functor (the right adjoint of an inclusion functor), and it follows that $\ab$-cellular approximation and $\ab$-nullification exist for any object of $\Derived_R$.

Similarly, when $\ab$ is a set there exists a left adjoint to the inclusion of the full subcategory of $\ab$-null objects, see for example Neeman's book \cite[Section 9]{NeemanBook}. This functor is in fact $\ab$-nullification and it is a \emph{localization} functor (the left adjoint to an inclusion functor).

One method for calculating $\ab$-cellular approximations is the formula given by Dwyer and Greenlees in~\cite{DwyerGreenlees}, which holds whenever $\ab=\{A\}$ and $A$ is a perfect complex. This was later generalized by Dwyer, Greenlees and Iyengar in~\cite{DwyerGreenleesIyengar}. A new method for calculating the $\ab$-cellular approximation for $R$-modules has been constructed by Benson in~\cite{Benson}, dubbed \emph{$k$-squeezed resolutions}. This method can be applied whenever $\ab$ is a set of simple modules and $R$ is an Artinian ring. One major benefit of Benson's construction is that it allows for explicit calculations.

As we will see, it is more natural to use Benson's method to construct the $\ab$-nullification of a module, rather than its $\ab$-cellular approximation. We generalize Benson's construction so that it applies whenever $\ab$ is a hereditary torsion class of modules. A \emph{hereditary torsion class} of modules is a class of modules that is closed under submodules, quotient modules, coproducts and extensions. The main result of this paper is the following.

\begin{theorem}
\label{the: Nullification as double endomorphism}
Let $\tc$ be a hereditary torsion class on left $R$-modules. For every left $R$-module $M$ there exists an injective left $R$-module $E$ such that the complex
\[\mathbf{R}\Hom_{\End_R(E)}(\Hom_R(M,E),E)\]
is a $\tc$-nullification of $M$. In particular, the differential graded algebra $\mathbf{R}\End_{\End_R(E)}(E)$ is a $\tc$-nullification of $R$.
\end{theorem}

The formula given in the abstract follows immediately from the distinguished triangle $\cell_\tc M \to M \to \Null_\tc M$ mentioned above.



The layout of this paper is as follows. The necessary background on hereditary torsion classes and the background on cellular approximations and nullifications is given in Section~\ref{sec: Background}. In Section~\ref{sec: The Construction} we describe the construction of nullification with respect to a hereditary torsion class and prove Theorem~\ref{the: Nullification as double endomorphism}. We study the case where $R$ is an Artinian ring in Section~\ref{sec: Torsion Theories and Cellular Approximation in Artinian Rings}. This section offers a different proof to a result of Benson (\cite[Theorem 5.1]{Benson}). Finally, Section~\ref{sec: Examples} provides several examples.

\subsection{Notation and Terminology}
By a ring we always mean an associative ring with a unit, not necessarily commutative. Unless otherwise noted all modules considered are left modules. A triangle always means an exact (distinguished) triangle in the unbounded derived category of left $R$-modules, denoted $\Derived_R$. A complex is always a chain-complex of $R$-modules. For complexes we use the standard convention that subscript grading is the negative of the superscript grading, i.e. $\square_{-i}=\square^i$. It is taken for granted that every $R$-module is a complex concentrated in degree 0 and with zero differential. A complex $X$ is \emph{bounded-above} if for some $n$ and for all $i>n$, $H_i(X)=0$. For complexes $X$ and $Y$ the notation $\Hom_R(X,Y)$ stands for the usual chain complex of homomorphisms. The notation $\mathbf{R}\Hom_R(-,-)$ stands for the derived functor of the $\Hom_R(-,-)$ functor. By $\End_R(M)$ we mean the endomorphisms ring of an $R$-module $M$. The symbol $\simeq$ stands for quasi-isomorphism of complexes.

\section{Background on Hereditary Torsion Theories and Cellular-Approximation, Nullification and Completion}
\label{sec: Background}

\subsection{Hereditary Torsion Theories}
\label{sub: Recollection about Hereditary Torsion Theories}
Below is a recollection of the definition and main properties of hereditary torsion theories. A thorough review of this material can be found in \cite{Stenstrom}.

\begin{definition}
A \emph{hereditary torsion class} $\tc$ is a class of $R$-modules that is closed under submodules, quotient modules, coproducts and extensions. Closure under extensions means that if $0 \to M_1 \to M_2 \to M_3 \to 0$ is a short exact sequence with $M_1$ and $M_3$ in $\tc$, then so is $M_2$. The modules in $\tc$ will be called \emph{$\tc$-torsion modules} (or just torsion modules when the torsion theory is clear from the context). The class of \emph{torsion-free} modules $\ff$ is the class of all modules $F$ satisfying $\Hom_R(C,F)=0$ for every $C \in \tc$. The pair $(\tc,\ff)$ is referred to as a \emph{hereditary torsion theory}. To every hereditary torsion theory there is an associated radical $t$, where $t(M)$ is the maximal torsion submodule of $M$. Note that $M/t(M)$ is therefore torsion-free.
\end{definition}

Every hereditary torsion class $\tc$ has an \emph{injective cogenerator} (see~\cite[VI.3.7]{Stenstrom}). This means there exists an injective module $E$ such that a module $M$ is torsion if and only if $\Hom_R(M,E)=0$. It is also important to note that in any hereditary torsion theory, the class of torsion-free modules is closed under injective hulls (see~\cite[VI.3.2]{Stenstrom}). Thus, if $F$ is a torsion-free module then the injective hull of $F$ is also torsion-free.

\begin{definition}
\label{def: Module of quotients}
Let $(\tc,\ff)$ be a hereditary torsion theory and let $t$ be the associated radical. An $R$-module $M$ is called \emph{$\ff$-closed} if, for every left ideal $\aideal \subset R$ such that $R/\aideal \in \tc$, the induced map $M=\Hom_R(R,M) \to \Hom_R(\aideal,M)$ is an isomorphism. The inclusion of the full subcategory of $\ff$-closed modules has a left adjoint $M \mapsto M_\ff$. The module $M_\ff$ is called the \emph{module of quotients} of $M$ (see \cite[IX.1]{Stenstrom}). The unit of this adjunction has the following properties: the kernel of the map $M \to M_\ff$ is $t(M)$, $M_\ff$ is torsion-free and the cokernel of this map is a torsion module.
\end{definition}

\subsection{Cellular-Approximation, Nullification and Completion}
The following recalls the basic properties of cellular approximation, as well as the definition of completion given by Dwyer and Greenlees in~\cite{DwyerGreenlees}.

\begin{definition}
\label{def: cellularization, nullification and completion}
Let $R$ be a ring and let $\ab$ be a class of $R$-complexes. We say an $R$-complex $X$ is \emph{$\ab$-complete} if $\ext_R^*(N,X)=0$ for any $\ab$-null object $N$. An $R$-complex $C$ is an \emph{$\ab$-completion} of $X$ if $C$ is $\ab$-complete and there is an $\ab$-equivalence $X \to C$. It is easy to see that an $\ab$-completion of a complex $X$ is unique up to an isomorphism in $\Derived_R$. As in~\cite{DwyerGreenlees}, we denote an $\ab$-completion of $X$ by $X^\wedge_\ab$.
\end{definition}

The following criterion for nullification is usually easier to check than the original definition. Its proof is easy and therefore omitted.
\begin{lemma}
\label{lem: Triangle criterion for nullification}
Let $R$ be a ring and let $\ab$ be a class of $R$-complexes. A complex $N$ is an \emph{$\ab$-nullification} of $X$ if there is a triangle $C \to X \to N$ such that $C$ is $\ab$-cellular and $N$ is $\ab$-null. In this case it also follows that $C$ an $A$-cellular approximation of $X$.
\end{lemma}

Recall that when $\ab$ is a set, the full subcategory of $\ab$-cellular objects of $\Derived_R$ is the localizing subcategory of generated by $A$. The \emph{localizing category generated by $\ab$}, denoted $\langle \ab \rangle$, is the smallest full triangulated subcategory of $\Derived_R$ that is closed under triangles, direct sums and retracts. Closure under triangles means that for every distinguished triangle in $\Derived_R$, if two of the objects are in the localizing subcategory, so is the third. The proof of the following lemma is clear.

\begin{lemma}
Let $\ab$ be a class of $R$-complexes, then every object of $\langle \ab \rangle$ is $\ab$-cellular. If $\bb$ is another class of $R$-complexes such that $\langle \ab \rangle= \langle \bb \rangle$, then $\ab$-cellular approximation is the same as $\bb$-cellular approximation.
\end{lemma}

\begin{remark}
In~\cite{DwyerGreenlees} $\ab$-cellular complexes were called \emph{$\ab$-torsion} while the term \emph{$\ab$-cellular} was reserved for complexes in $\langle\ab\rangle$. When $\tc$ is a hereditary torsion theory, the two terms agree (by Lemma~\ref{lem: Every torsion is cellular} below).
\end{remark}

\subsection{Cellular-Approximation with respect to a Hereditary Torsion Theory}
\label{sec: Cellular-Approximation with respect to a Hereditary Torsion Theory}

Let $\tc$ be a hereditary torsion class. It is not immediately apparent that $\tc$-cellular approximation exists. Below, in Lemma~\ref{lem: Every torsion is cellular}, we show that $\langle\tc\rangle$ is the same as the localizing subcategory generated by a set $\ab_\tc$. This immediately implies that $\tc$-cellular approximation and $\tc$-nullification exist for any $R$-complex, see Corollary~\ref{cor: Consequence of being generated by AT}.

\begin{definition}
Given a hereditary torsion class of $R$-modules $\tc$, we denote by $\ab_\tc$ the set of all cyclic $\tc$-torsion modules.
\end{definition}

\begin{lemma}
\label{lem: Every torsion is cellular}
Let $\tc$ be a hereditary torsion class, then every $\tc$-torsion module is $\ab_\tc$-cellular and hence $\langle\tc\rangle=\langle\ab_\tc\rangle$.
\end{lemma}
\begin{proof}
Clearly, every cyclic $\tc$-torsion module is $\ab_\tc$-cellular. Therefore every direct sum of cyclic $\tc$-torsion modules is $\ab_\tc$-cellular. Let $M$ be a $\tc$-torsion module, then there is a surjection $C(M)=\oplus_{m\in M} R/\mathrm{ann}(m) \twoheadrightarrow M$. Since every hereditary torsion theory is closed under submodules, $R/\mathrm{ann}(m)$ is $\tc$-torsion for every $m\in M$. Clearly $C(M)$ is $\ab_\tc$-cellular and $\tc$-torsion. Next we build a resolution $X$ of $M$ using $\ab_\tc$-cellular modules. Let $X_0=C(M)$, and let $d_0:X_0 \twoheadrightarrow M$ the map defined above. The kernel of $d_0$ is $\tc$-torsion, so there is an epimorphism $C(\ker(d_0)) \twoheadrightarrow \ker(d_0)$. Let $X_1=C(\ker(d_0))$ and let $d_1$ be the composition $X_1 \twoheadrightarrow \ker(d_0) \hookrightarrow X_0$. In this way $X$ is built inductively and it is clear that $X$ is quasi-isomorphic to $M$. By construction, $X$ is in the localizing subcategory generated by $\ab_\tc$.
\end{proof}

\begin{corollary}
\label{cor: Consequence of being generated by AT}
Let $\tc$ be a hereditary torsion class, then $\tc$-cellular approximation, $\tc$-nullification and $\tc$-completion exist for every complex. Moreover, a complex $X$ is $\tc$-cellular if and only if $X\in\langle\tc\rangle$.
\end{corollary}
\begin{proof}
Lemma~\ref{lem: Every torsion is cellular} implies that $\tc$-cellular approximation is the same as $\ab_\tc$-cellular approximation. As mentioned in Section~\ref{sec: Introduction}, $\ab_\tc$-cellular approximation exists for every complex. The proof of the other claims is similar.
\end{proof}

\begin{lemma}
\label{lem: Bounded-above is cellular}
Let $\tc$ be a hereditary torsion class.
\begin{enumerate}
\item If $X$ is a $\tc$-cellular complex then the homology groups of $X$ are $\tc$-torsion $R$-modules.
\item If $X$ is a bounded-above complex such that the homology groups of $X$ are $\tc$-torsion then $X$ is $\tc$-cellular.
\end{enumerate}
\end{lemma}
\begin{proof}
Let $\cc$ be the full subcategory of $\Derived_R$ containing all objects whose homology groups are $\tc$-torsion $R$-modules. The properties of a hereditary torsion theory show that $\cc$ is localizing subcategory. Since $\cc$ contains $\tc$, then $\cc$ also contains $\langle \tc \rangle$. This proves the first statement.

Now suppose $X$ is a bounded-above complex and that $H_i X \in \tc$ for all $i$. Because $X$ is bounded-above, $X$ belongs to the localizing subcategory generated by the homology groups of $X$ (see for example~\cite[5.2]{DwyerGreenlees}). Since the homology groups of $X$ all belong to $\langle \tc \rangle$, so does $X$.
\end{proof}

\begin{remark}
If $R$ is a commutative Noetherian ring, then a complex $X$ is $\tc$-cellular if and only if all the homology groups of $X$ are $\tc$-torsion. This easily follows from a result of Neeman~\cite[Theorem 2.8]{NeemanChromatic}. However, Example~\ref{exa: Noncommutative} shows a noncommutative ring $R$ and a complex $X$ such that $H_i(X)$ is $\tc$-torsion for all $i$ but $X$ is not $\tc$-cellular.
\end{remark}

\section{Nullification Construction}
\label{sec: The Construction}

In \cite{Benson}, Benson gives a construction called \emph{$k$-squeezed resolution} which yields $k$-cellular approximations over the ring $kG$, where $k$ is a prime field and $G$ is a finite group. We generalize Benson's construction so as to produce $\tc$-cellular approximations over any ring $R$, where $\tc$ is a hereditary torsion class. In fact, we give two isomorphic constructions.

\begin{construction}
\label{Construction}
Let $(\tc,\ff)$ be a hereditary torsion theory with radical $t$. For an $R$-module $M$ we construct the $\tc$-nullification of $M$ as a cochain complex $I^0 \xrightarrow{d} I^1 \xrightarrow{d} I^2 \xrightarrow{d} \cdots$ inductively.

Let $M^0=M$, let $F^0=M^0/t(M^0)$ and let $I^0$ be the injective hull of $F^0$. Note that since $F^0$ is torsion-free, so is $I^0$. We proceed by induction, set
\[M^{n+1}=I^n/F^n, \qquad F^{n+1}=M^{n+1}/t(M^{n+1})\]
and let $I^{n+1}$ be the injective hull of $F^{n+1}$. Again $I^{n+1}$ is torsion-free because $F^{n+1}$ is. The differential $d:I^n \to I^{n+1}$ is the composition $I^n \to M^{n+1} \to F^{n+1} \to I^{n+1}$. The image of $d:I^n \to I^{n+1}$ is $F^{n+1}$ and therefore $d\circ d=0$. Denote the resulting complex by $I$. The natural map $M \to I^0$ extends to a map of complexes $M\to I$.
\end{construction}

\begin{construction}
\label{Construction2}
For an $R$-module $M$ we construct a cochain complex $J^0 \xrightarrow{d^0} J^1 \xrightarrow{d^1} J^2 \xrightarrow{d^2} \cdots$ inductively.

Let $Q^0=M$, let $N^0=(Q^0)_\ff$ and let $J^0$ be the injective hull of $N^0$. Denote by $d^{-1}$ the map $M \to J^0$. Now proceed by induction, let
\[Q^{n+1}=J^n/\im(d^{n-1}), \qquad N^{n+1}=(Q^{n+1})_\ff\]
and let $J^{n+1}$ be the injective hull of $J^n$. The differential $d^n:J^n \to J^{n+1}$ is the composition $J^n \to Q^{n+1} \to N^{n+1} \to J^{n+1}$. Clearly, $d^{n+1}\circ d^n=0$. Denote the resulting complex by $J$. The natural map $M \to J^0$ extends to a map of complexes $M\to J$. Note that for every $n$, $J^n$ is torsion-free because $N^n$ is.
\end{construction}

\begin{lemma}
Let $J$ be the complex constructed from $M$ in~\ref{Construction2}, then $H_0(J) \cong M_\ff$.
\end{lemma}
\begin{proof}
It easily follows from the definition of an $\ff$-closed module that any injective torsion-free module is $\ff$-closed, therefore $J^0$ is $\ff$-closed. For any $\ff$-closed module $K$ there is an isomorphism $K\cong K_\ff$ (see~\cite[page 198]{Stenstrom}), therefore $(J^0)_\ff\cong J^0$ and $(M_\ff)_\ff\cong M_\ff$.

The module of quotients functor is left exact (see~\cite[page 199]{Stenstrom}). Hence applying the module of quotients functor to the sequence $M_\ff \to J^0 \to J^0/M_\ff$ yields an exact sequence:
\[ 0 \to M_\ff \to J^0 \to (J^0/M_\ff)_\ff\]
We see that $J^0/M_\ff$ is torsion-free, because it is isomorphic to a submodule of the torsion-free module $(J^0/M_\ff)_\ff$.

Now consider the short exact sequence
\[ M_\ff/\im(M) \to Q^1 \to J^0/M_\ff\]
The module $M_\ff/\im(M)$ is a torsion module (see Definition~\ref{def: Module of quotients}), while the module $J^0/M_\ff$ is torsion free. From the definition of the radical $t$ it follows that $M_\ff/\im(M) \cong t(Q^1)$. Therefore $M_\ff$ is the kernel of $J^0 \to N^1$ and the proof is complete.
\end{proof}

\begin{lemma}
\label{lem: Construction of Nullification}
Let $M$ be an $R$-module, let $I$ be the complex constructed from $M$ in~\ref{Construction}, let $C$ be a complex such that $C \to M \to I$ is a distinguished triangle and let $J$ be the complex constructed from $M$ in~\ref{Construction2}. Then $C$ is a $\tc$-cellular approximation of $M$ and both $I$ and $J$ are $\tc$-nullifications of $M$. In particular, $H_0(\Null_\tc M) \cong M_\ff$ for any $R$-module $M$.
\end{lemma}
\begin{proof}
We can choose $C$ to be the complex $M \to I^0 \to I^1 \to \cdots$ with $M$ in degree 0. The homology of $C$ is easy to compute: $H^n(C)=t(M^n)$, with $M^n$ as defined in the Nullification Construction~\ref{Construction} above. By Lemma~\ref{lem: Bounded-above is cellular}, the complex $C$ is $\tc$-cellular. The complex $I$ is $\tc$-null, simply because $I$ is composed of torsion-free injective modules. Thus, by Lemma~\ref{lem: Triangle criterion for nullification}, $I$ is a $\tc$-nullification of $M$ and $C$ is a $\tc$-cellular approximation of $M$.

Similar reasoning shows that $J$ is a $\tc$-nullification of $M$. The complex $J$ is $\tc$-null, simply because $J$ is composed of torsion-free injective modules. The homology of $J$ is: $H^n(J)=t(Q^n)$ for $n>0$ and $H^0(J)=M_\ff$. Let $C'$ be a complex such that there is a distinguished triangle $C' \to M \to J$. The long exact sequence in homology yields: $H^0(C')=t(M)$, $H^1(C')= M_\ff/M$ and $H^n(C')=H^{n-1}(J)$ for $n>1$. Note that $M_\ff/M$ is a $\tc$-torsion module. By Lemma~\ref{lem: Bounded-above is cellular}, the complex $C'$ is $\tc$-cellular and hence $J$ is a $\tc$-nullification of $M$ and $C'$ is a $\tc$-cellular approximation of $M$. In particular, $H_0(\Null_\tc M) = H_0(J) \cong M_\ff$.
\end{proof}

\begin{remark}
It follows that the complexes $I$ and $J$ in Lemma~\ref{lem: Construction of Nullification} are isomorphic in the derived category of $R$. In fact, they are isomorphic as complexes. To construct this isomorphism one needs the following property: for any $R$-module $L$ the injective hull of $L_\ff$ and the injective-hull of $L/t(L)$ are the same; this is because $(L/t(L))_\ff = L_\ff$ and $L/t(L)$ is an essential submodule of $(L/t(L))_\ff$ (see~\cite[IX.2.4]{Stenstrom}). Using the aforementioned property it is a simple exercise to construct the isomorphism inductively.
\end{remark}

Using the construction above we can give a different description of $\tc$-nullification, the one shown in Theorem~\ref{the: Nullification as double endomorphism}. Before proving Theorem~\ref{the: Nullification as double endomorphism} it is necessary to note some properties of the functor $\Hom_R(-,E)$.

Let $E$ be an $R$-module and let $\ee$ be the endomorphism ring $\End_R(E)=\Hom_R(E,E)$. The functor $\Hom_R(-,E)$ is a contravariant functor from left $R$-complexes to left $\ee$-complexes. This left $\ee$-action is simply composition on the left with the morphisms in $\ee$. In other words, the left $\ee$-action on $\Hom_R(-,E)$ is induced by the left $\ee$-action on $E$ itself. Moreover, the functor $\Hom_{\ee}(-,E)$ is a contravariant functor, this time from left $\ee$-complexes to left $R$-complexes. Here the left $R$-action on $\Hom_\ee(-,E)$ comes from the left $R$-action on $E$ (which commutes with the left $\ee$-action on $E$). In particular, there is a derived version of this functor: $\mathbf{R}\Hom_\ee(-,E): \Derived_\ee \to \Derived_R$.

\begin{proof}[Proof of Theorem~\ref{the: Nullification as double endomorphism}]
Given an $R$-module $M$, construct a $\tc$-nullification of $M$ in the way prescribed in~\ref{Construction}. This construction results in a cochain complex $I$, with $I^n$ being an injective torsion-free module. Let $E$ be a torsion-free injective $R$-module such that for every $n$, $I^n$ is a direct summand of a finite direct sum of copies of $E$. For example, one can take $E$ to be the product $\prod_n I^n$. Denote by $\ee$ the endomorphism ring $\End_R(E)$.

Now consider the triangle $C \to M \to I$. Since $I$ is a $\tc$-nullification of $M$, $C$ is a $\tc$-cellular approximation of $M$. Applying the functor $\Hom_R(-,E)$ to this triangle yields a triangle in $\Derived_\ee$:
\[ \Hom_R(I,E) \to \Hom_R(M,E) \to \Hom_R(C,E) \]
Since $E$ is injective, $H_i(\Hom_R(C,E)) \cong \Hom_R(H_i (C),E)$. Since the homology groups of $C$ are torsion, $\Hom_R(H_i (C),E)=0$. Therefore the map $\Hom_R(I,E) \to \Hom_R(M,E)$ is a quasi-isomorphism of $\ee$-complexes.

Because $I^n$ is a direct summand of a finite direct sum of copies of $E$, the $\ee$-module $\Hom_R(I^n,E)$ is projective. Thus the map $\Hom_R(I,E) \to \Hom_R(M,E)$ is a projective resolution of $\Hom_R(M,E)$ in the category of $\ee$-modules. We conclude that the complex $\Hom_\ee(\Hom_R(I,E),E)$ is the derived functor $\mathbf{R} Hom_\ee(\Hom_R(M,E),E)$.

Because $I^n$ is a direct summand of a finite direct sum of copies of $E$, one readily sees that the $R$-module \[\Hom_\ee(\Hom_R(I^n,E),E)\] is naturally isomorphic to $I^n$ and therefore $\Hom_\ee(\Hom_R(I,E),E)\cong I$.
\end{proof}

\begin{remark}
As noted in Theorem~\ref{the: Nullification as double endomorphism}, $\Null_\tc R \simeq \mathbf{R}\End_{\ee}(E)$ and therefore
\[ R_\ff \cong H^0 (\Null_\tc R) \cong H^0 (\mathbf{R}\End_{\ee}(E)) = \End_{\ee}(E)\]
This isomorphism recovers~\cite[IX.3.3]{Stenstrom}, where it is stated that there is an injective $R$-module $E$ such that $R_\ff \cong \End_{\End_R(E)}(E)$. Also note that $\Null_\tc R$ is quasi-isomorphic to a differential graded algebra. This also follows from a result of Dwyer~\cite[Proposition 2.5]{DwyerNoncommutative}, where it is shown that for any set of complexes $\ab$, the complex $\Null_\ab R$ is quasi-isomorphic to a differential graded algebra.
\end{remark}

\begin{remark}
Let $X= \cdots \to X_n \to X_{n-1} \to \cdots$ be a complex such that there exists some $m$ for which $X_n=0$ for all $n>m$. Then it is possible to generalize the Nullification Construction~\ref{Construction} to give the $\tc$-nullification of $X$. Moreover, this generalized construction of $\Null_\tc X$ can be done in such a way that for $n>m$ $(\Null_\tc X)_n=0$, while for $n\leq m$ $(\Null_\tc X)_n$ is a finite direct sum of torsion-free injective modules. Therefore $(\Null_\tc X)_n$ is itself a torsion-free injective for $n\leq m$. Now it is easy to see that the proof of Theorem~\ref{the: Nullification as double endomorphism} works for $\Null_\tc X$ as well and yields the same result. Namely, there exists an injective $R$-module $E$ such that
\[ \Null_\tc X \simeq \mathbf{R}\Hom_{\End_R(E)}(\Hom_R(X,E),E) \]
Clearly, this result carries over to any bounded-above complex $X$.
\end{remark}

Say an injective module $E$ is \emph{sufficient to compute the $\tc$-nullification of $M$} if
\[\Null_\tc M \simeq \mathbf{R} Hom_\ee(\Hom_R(M,E),E)\]
where $\ee=\End_R(M)$. Given an $R$-module $M$ one can use the proof of Theorem~\ref{the: Nullification as double endomorphism} to construct an injective module $E$ which is sufficient to compute the $\tc$-nullification of $M$. However there are other injective modules sufficient to compute the $\tc$-nullification of $M$, as shown by the following proposition.
\begin{proposition}
Let $M$ be an $R$-module and let $E$ be an injective cogenerator of $\tc$. Denote by $\ee$ the ring $\End_R(E)$.
\begin{enumerate}
\item If the $\ee$-module $\Hom_R(M,E)$ has a resolution composed of finitely generated projective modules in each degree, then $E$ is sufficient to compute the $\tc$-nullification of $M$.
\item There exists an ordinal $\alpha$ such that the module $E'=\prod_{i<\alpha} E$ is sufficient to compute the $\tc$-nullification of $M$.
\end{enumerate}
\end{proposition}
\begin{proof}
Let $P$ be a finitely generated projective $\ee$-module, then it is easy to see that $\Hom_R(\Hom_\ee(P,E),E)$ is naturally isomorphic to $P$. Now let $F$ be a projective resolution of $\Hom_R(M,E)$ over $\ee$ and assume $F$ is composed of finitely generated projective modules in each degree. Then $\Hom_R(\Hom_\ee(F,E),E)$ is naturally isomorphic to $F$.

The quasi-isomorphism $\eta:F \to \Hom_R(M,E)$ induces a map $\Hom_\ee(\Hom_R(M,E),E) \to \Hom_\ee(F,E)$. Composing with the natural map $M \to \Hom_\ee(\Hom_R(M,E),E)$ yields a map $\mu:M \to \Hom_\ee(F,E)$. It is easy to see that $\Hom_R(\mu,E)$ is the quasi-isomorphism $\eta$.

Consider the triangle $C \to M \xrightarrow{\mu} \Hom_\ee(F,E)$. Clearly, $\Hom_\ee(F,E)$ is $\tc$-null. Since $\Hom_R(\mu,E)$ is a quasi-isomorphism, $\Hom_R(C,E)$ is quasi-isomorphic to zero. This implies $\Hom_R(H_i(C),E)=0$ for all $i$. Since $E$ is an injective cogenerator for $\tc$, $H_i(C)$ is torsion for all $i$. Clearly $C$ is bounded-above and so, by Lemma~\ref{lem: Bounded-above is cellular}, $C$ is $\tc$-cellular. We conclude that $\Hom_\ee(F,E)$ is a $\tc$-nullfication of $M$ and $E$ is sufficient to compute the $\tc$-nullification of $M$.

We now turn our attention to the second item in the proposition. By~\cite[VI.3.9]{Stenstrom}, every torsion-free module has a monomorphism to some direct product of copies of $E$. In particular, every torsion-free injective is a isomorphic to a direct summand of some direct product of copies of $E$.

Let $I$ be the complex described in the Nullification Construction~\ref{Construction}. Let $E'$ be a direct product of copies of $E$ such that for every $n$, $I^n$ is isomorphic to a direct summand of $E'$. Clearly the $\End_R(E')$-complex $\Hom_R(I,E')$ is a projective resolution of $\Hom_R(M,E')$ which is composed of finitely generated projective modules in every degree. Hence $E'$ is sufficient to compute the $\tc$-nullification of $M$.
\end{proof}

\section{Torsion Theories and Cellular Approximation in Artinian Rings}
\label{sec: Torsion Theories and Cellular Approximation in Artinian Rings}
Throughout this section $R$ is an Artinian ring and $\simples$ is a set of non-isomorphic simple modules of $R$. Define a class $\ff$ of $R$-modules by $\ff=\{F \ |\ \Hom_R(S,F)=0 \text{ for all } S \in \simples\}$
and define a class $\tc$ by setting $ \tc=\{M \ |\ \Hom_R(M,F)=0 \text{ for all } F \in \ff\}$. By~\cite[VIII.3]{Stenstrom}, the pair $(\tc,\ff)$ forms a hereditary torsion theory (alternatively, one can easily deduce this from Lemma~\ref{lem: Cogenerator for torsion of simples} below). Because $R$ is Artinian, every hereditary torsion theory of $R$-modules is generated by a set of simple modules (see~\cite[VIII]{Stenstrom}), so this context covers all hereditary torsion theories over $R$. In this section we give several results regarding $\tc$-nullification. We also give a different proof for a result of Benson \cite[Theorem 5.1]{Benson} in Corollary~\ref{cor: Bensons result}.

Let $\Omega$ be the set of isomorphism classes of simple modules of $R$ and let $\simples'$ be the complement of $\simples$ in $\Omega$. We denote by $E$ the product of the injective hulls of the simple modules in $\simples'$ and denote by $P$ the direct sum of the projective covers of those simple modules. We show that $E$ is an injective cogenerator of $\tc$ and that being $\tc$-cellular is the same as the being $\simples$-cellular.

\begin{lemma}
\label{lem: Cyclic E-complete is simples cellular}
Let $C$ be a cyclic $R$-module such that $\Hom_R(C,E)=0$, then $C$ is $\simples$-cellular.
\end{lemma}
\begin{proof}
Since $R$ is Artinian, $C$ admits a composition series
\[ 0 = C_0 \subset C_1 \subset \cdots \subset C_m=C \ , \]
where all the quotients $C_i/C_{i-1}$ are simple modules. We next show that $C_i/C_{i-1} \in \simples$ for all $i$. Suppose that for some $i$, $C_i/C_{i-1}\cong S'$ for some $S'\in \simples'$. Let $x \in C_i \smallsetminus C_{i-1}$, then the cyclic module generated by $x$ has $S'$ as a quotient. This implies the submodule $Rx$ of $C$ has a non-zero map to $E(S')$ - the injective hull of $S'$. Clearly such a map can be lifted to a non-zero map $C\to E$, in contradiction. Therefore $C_i/C_{i-1}\cong S$ for some $S\in \simples$. Now a simple inductive argument on $i$ shows that $C_i \in \langle\simples\rangle$ for every $i$, and hence $C$ is $\simples$-cellular.
\end{proof}

\begin{corollary}
\label{cor: tc-cell is simples-cell}
A complex $X$ is $\tc$-cellular if and only if $X$ is $\simples$-cellular.
\end{corollary}
\begin{proof}
We need to show that $\langle \tc \rangle = \langle \simples \rangle$. Since $\simples \subset \tc$, we only need to show that $\tc \subset \langle \simples \rangle$. By Lemma~\ref{lem: Every torsion is cellular} it is enough to show that every cyclic $R$-module is $\simples$-cellular, but that is immediate from Lemma~\ref{lem: Cyclic E-complete is simples cellular}.
\end{proof}

\begin{lemma}
\label{lem: Cogenerator for torsion of simples}
The module $E$ is an injective cogenerator for $\tc$.
\end{lemma}
\begin{proof}
Let $\mathcal{U}$ be the class of modules $M$ such that $\Hom_R(M,E)=0$. Then $\mathcal{U}$ is a hereditary torsion theory. Because $\Hom_R(S,S')=0$ for every  $S\in \simples$ and $S' \in \simples'$, we see that $\Hom_R(S,E(S'))=0$, where $E(S')$ is the injective envelope of $S'$. Hence $E \in \ff$ and therefore $\mathcal{U}$ contains $\tc$.

Next, let $M$ be in $\mathcal{U}$. To show that $M$ is a $\tc$-torsion module it is enough to show that every cyclic submodule of $M$ is a torsion module, because $M$ is a quotient of the direct sum of its cyclic submodules. So let $C$ be a cyclic submodule of $M$. Since $E$ is injective, it follows that $\Hom_R(C,E)=0$. By Lemma~\ref{lem: Cyclic E-complete is simples cellular} $C$ is $\simples$-cellular. Therefore $C$ is $\tc$-cellular and by Lemma~\ref{lem: Bounded-above is cellular} $C$ is $\tc$-torsion.
\end{proof}

\begin{lemma}
\label{lem: P-null is E-conull}
For any complex $X$, $\ext_R^*(P,X)=0$ if and only if $\ext^*_R(X,E)=0$.
\end{lemma}
\begin{proof}
This is known when $X$ is a finitely generated $R$-module, see Benson's book~\cite[1.7.6 \&1.7.7]{BensonBookI}. Now suppose $X$ is any $R$-module. Since $P$ is a finitely generated projective module, $\Hom_R(P,X)=0$ if and only if $\Hom_R(P,X')=0$ for every finitely generated submodule $X'$ of $X$. Similarly, because $E$ is injective, $\Hom_R(X,E)=0$ if and only if $\Hom_R(X',E)=0$ for every finitely generated submodule $X'$ of $X$. Hence the lemma holds for any $R$-module. Finally, let $X$ be any complex, then $\ext^*_R(P,X)=\Hom_R(P,H_*(X))$. Similarly $\ext^*_R(X,E)=\Hom_R(H^*(X),E)$.
\end{proof}

\begin{corollary}
\label{cor: Bensons result}
For any $R$-module $M$, a $\tc$-nullification of $M$ is also a $P$-completion of $M$ and is therefore given by
\[ \Null_\tc M \simeq \mathbf{R} \Hom_{\End_R(P)} (\Hom_R(P,R),\Hom_R(P,M))\]
\end{corollary}
\begin{proof}
Consider the triangle $\cell_\tc M \to M \xrightarrow{\nu} \Null_\tc M$. Lemma~\ref{lem: Cogenerator for torsion of simples} implies that $E$ is $\tc$-null and therefore $\ext_R^*(\cell_\tc M,E)=0$. By Lemma~\ref{lem: P-null is E-conull}, $\cell_\tc M$ is $P$-null and $\nu$ is a $P$-equivalence.

It remains to show that $\Null_\tc M$ is $P$-complete. Let $I$ be the $\tc$-nullification of $M$ described in~\ref{Construction}. The full subcategory of $P$-complete objects in $\Derived_R$ is closed under isomorphisms, completion of triangles, products and retracts. Denote this subcategory by $\cc$. From Lemma~\ref{lem: P-null is E-conull}, we see that $E \in \cc$ and therefore every product of $E$ is also in $\cc$.  Lemma~\ref{lem: Cogenerator for torsion of simples} and~\cite[VI.3.9]{Stenstrom} imply that every torsion-free module is a submodule of a product of copies of $E$. Since $I^n$ is injective, it is a direct summand of some product of copies of $E$, hence $I^n$ is also an object of $\cc$.

Let $I(n)$ denote the cochain complex $I^0 \to I^1 \to \cdots \to I^n$. An inductive argument shows that $I(n) \in \cc$. There is a triangle $ I \to \prod_n I(n) \xrightarrow{\phi-1} \prod_n I(n)$, where the map $\phi$ is induced by the maps $I(n+1)\to I(n)$. Hence $I$ is $P$-complete.

By Dwyer and Greenlees \cite[Theorem 2.1]{DwyerGreenlees}, the $P$-completion of an $R$-module $M$ is given by
\[ M^\wedge_P \simeq \mathbf{R} \Hom_{\End_R(P)} (\Hom_R(P,R),\Hom_R(P,M))\]
\end{proof}

Corollary~\ref{cor: Bensons result} above implies Benson's formula for $\tc$-cellular approximation given in~\cite[Theorem 5.1]{Benson}. This corollary also explains the connection between Benson's formula and Dwyer and Greenlees formula for $P$-completion from~\cite[Theorem 2.1]{DwyerGreenlees}.

\section{Examples}
\label{sec: Examples}

\begin{example}
\label{exa: Commutative example}
Let $I$ be a two-sided ideal of $R$ such that $I$ is finitely generated as a left $R$-module. An $R$-module $M$ will be called \emph{$I$-torsion} if for every $m\in M$ there exists some $n$ such that $I^n m=0$. It is not difficult to show that the class of $I$-torsion modules forms a hereditary torsion class $\tc$ (see \cite[VI.6.10]{Stenstrom}). Using Lemma~\ref{lem: Every torsion is cellular} it is easy to conclude that $\langle\tc\rangle=\langle R/I \rangle$. Hence $\tc$-cellular approximation is the same as $R/I$-cellular approximation and the same goes for nullification. Note that in this case the radical $t$ associated with $\tc$ has a simple description: for any $R$-module $M$
\[t(M)=\colim_{n\to \infty} \Hom_R(R/I^n,M)\]

Now suppose $R$ is a commutative Noetherian ring. Dwyer and Greenlees have shown in~\cite{DwyerGreenlees} that $R/I$-cellular approximation computes $I$-local cohomology, namely that there is a natural isomorphism $H_I^*(M)\cong H^*(\cell_{R/I} M)$. Recall there is an isomorphism:
\[ H_I^*(M) \cong \colim_{n\to \infty} \ext_R^*(R/I^n,M)\]
These facts show that $\tc$-cellular approximation is the derived functor of the radical $t$. Moreover, in this case an object $X\in \Derived_R$ is $\tc$-cellular if and only $H_n(X)$ is $\tc$-torsion for all $n$, see \cite[6.12]{DwyerGreenlees}.
\end{example}

\begin{example}
\label{exa: Noncommutative}
Here is an example of a case where $\tc$-cellular approximation is not the derived functor of the associated radical. Let $G$ be the symmetric group on 3 elements, let $k$ be the field $\Int/3\Int$ and let $R$ be the group ring $k[G]$. There is an augmentation map $R \to k$, where $k$ has the trivial $G$-action. Let $I$ be the augmentation ideal. As before, denote the class of $I$-torsion modules by $\tc$ and the associated radical by $t$. Since $R$ is an Artinian ring, the sequence $I\supseteq I^2 \supseteq I^3 \supseteq \cdots$ stabilizes. So there is a fixed index $m$ such that $t(M)=\Hom_R(R/I^m,M)$ for every $R$-module $M$. Therefore, the derived functors of the torsion radical $t$ are the functors $\ext_R^*(R/I^m,-)$. In particular, $\ext_R^i(R/I^m,R)=0$ for all $i>0$, because $R$ is injective. On the other hand, a calculation using Benson's methods from~\cite{Benson} shows that $H^n(\cell_\tc R)$ is non-zero for infinitely many values of $n$; thereby showing that $\cell_\tc$ is not the derived functor of $t$. We describe this calculation next.

From the surjection $G \to \Int/2$ one sees that $R$ has two simple modules, the trivial module $k$ and a one dimensional simple module $\omega$. As a left module, $R \cong E_k\oplus E_\omega$ where $E_k$ and $E_\omega$ are the injective hulls of $k$ and $\omega$ respectively. The module $E_k$ has a composition series
\[k \subset B \subset E_k, \quad\text{where } B/k\cong \omega \ \text{ and }\ E_k/B\cong k\]
The composition series for $E_\omega$ is
\[\omega \subset B' \subset E_\omega, \quad\text{where } B'/\omega\cong k \ \text{ and }\ E_\omega/B'\cong \omega\]
In addition $E_\omega/\omega \cong B$ and $E_k/k \cong B'$. Since $E_\omega$ is $k$-null (see Lemma~\ref{lem: Cogenerator for torsion of simples}), then
\[\Null_\tc R \simeq \Null_\tc E_\omega \oplus \Null_\tc E_k \cong E_\omega \oplus \Null_\tc E_k\]
So we need only compute $\Null_\tc E_k$. Applying Construction~\ref{Construction} to the module $E_k$ we get the complex $I$ which is $E_\omega \xrightarrow{d} E_\omega \xrightarrow{d} E_\omega \xrightarrow{d} \cdots$ where $d$ is the composition $E_\omega \twoheadrightarrow \omega \hookrightarrow E_\omega$. Hence $H^n(\Null_\tc E_k)=k$ for $n>1$ and therefore $H^n(\cell_\tc R)$ is non-zero for infinitely many values of $n$. In fact
\[ H^n(\cell_\tc R)=\left\{
     \begin{array}{ll}
       k, & n=0; \\
       0, & n=1; \\
       k, & n>1.
     \end{array}
   \right.\]

It is important to note that in this case, a complex $X$ such that $H_n(X)$ is $\tc$-torsion for all $n$ need not be $\tc$-cellular. Consider, for example, the complex $R^\wedge_\tc$. As we explain below, the homology groups $H_n(R^\wedge_\tc)$ are $\tc$-torsion for all $n$. On the other hand, the $\tc$-equivalences $\cell_\tc R \to R$ and $R \to R^\wedge_\tc$ show that $\cell_\tc R$ is $\tc$-equivalent to $R^\wedge_\tc$. If $R^\wedge_\tc$ was $\tc$-cellular, then $R^\wedge_\tc$ would have been quasi-isomorphic to $\cell_\tc R$, because a $\tc$-equivalence between $\tc$-cellular complexes is a quasi-isomorphism. As we show next, the complex $R^\wedge_\tc$ has no homology in negative degrees and so cannot be quasi-isomorphic to $\cell_\tc R$.

It remains to explain the properties of $R^\wedge_\tc$ used above. From Corollary~\ref{cor: tc-cell is simples-cell} we learn that $R^\wedge_\tc \simeq R^\wedge_k$ and $\cell_\tc R\simeq \cell_k R$. Without going into details, combining \cite[5.9]{DwyerGreenleesIyengar} with \cite[4.3]{DwyerGreenlees} shows that
\[ R^\wedge_k \simeq \mathbf{R} \Hom_R(\cell_k R, R)\]
This immediately implies that $R^\wedge_\tc$ has no homology in negative degrees. We next show that $H_n(R^\wedge_\tc)$ is $\tc$-torsion for all $n$. Since $E_\omega$ is a $\tc$-null module, $\ext^*_R(E_\omega,R^\wedge_\tc)=0$. Recall that $R$ is a group-algebra and therefore $E_\omega$ is also the projective cover of $\omega$. Because $E_\omega$ is projective we have
\[ \ext^{-n}_R(E_\omega, R^\wedge_\tc) \cong \Hom_R(E_\omega, H_n(R^\wedge_\tc))\]
Hence, by Lemma~\ref{lem: P-null is E-conull}, $\ext^*_R(H_n(R^\wedge_\tc),E_\omega)=0$. Lemma~\ref{lem: Cogenerator for torsion of simples} shows $E_\omega$ is an injective cogenerator for $\tc$, therefore $H_n(R^\wedge_\tc)$ is $\tc$-torsion.
\end{example}

\begin{example}
This example relates $\tc$-nullification with Cohn localization. We begin by recalling the definition of Cohn localization. Let $S=\{f_\alpha:P_\alpha \to Q_\alpha\}$ be a set of maps between finitely generated projective $R$-modules. Say a ring map $R\to R'$ is \emph{$S$-inverting} if $\Hom_R(f,R')$ is an isomorphism for every $f \in S$. A \emph{Cohn localization} of $R$ with respect to $S$ is a ring map $R\to S^{-1}R$ which is initial among all $S$-inverting ring maps. Note that the definition given here is not the standard definition (see e.g.~\cite{DwyerNoncommutative}), but it is equivalent to the standard one.

Let $\cc_S$ be the set of cones of the maps $f_\alpha$. In~\cite{DwyerNoncommutative}, Dwyer considers $\cc_S$-nullification and shows that $H_0(\Null_{\cc_S}R)=S^{-1}R$ (see~\cite[3.2]{DwyerNoncommutative}). Combining Dwyer's results with Theorem~\ref{the: Nullification as double endomorphism} yields the following proposition.

\begin{proposition}
\label{pro: Cohn localization}
Let $\tc$ be a hereditary torsion-class of $R$-modules. If $\langle\tc\rangle=\langle\cc_S\rangle$ for some set of maps $S$ between finitely generated projective $R$-modules, then
\begin{enumerate}
\item $\Null_\tc(-) \simeq (-)_\ff$,
\item the module of quotients functor $(-)_\ff$ is exact and
\item there is an isomorphism $S^{-1}R\otimes_R M \cong M_\ff$ for every module $M$.
\end{enumerate}
\end{proposition}
\begin{proof}
By Theorem~\ref{the: Nullification as double endomorphism}, for every $R$-module $M$ the complex $\Null_\tc M$ has no homology in positive degrees. By~\cite[Proposition 3.1]{DwyerNoncommutative}, $\Null_\tc R$ has no homology in negative degrees. Moreover, a result of Miller \cite{MillerFiniteLocalizations} (see also~\cite[Proposition 2.10]{DwyerNoncommutative}) shows that for every $R$-module $M$, $\Null_\tc M \simeq \Null_\tc R \otimes_R^\mathbf{L} M$. This implies that $\Null_\tc M$ has no homology in negative degrees.

We conclude that for every $R$-module $M$, $\Null_\tc M$ has homology only in degree zero and therefore, by Lemma~\ref{lem: Construction of Nullification}, $\Null_\tc M$ is quasi-isomorphic to $M_\ff$. Since the functor $\Null_\tc$ is exact, so is $(-)_\ff$. Since $\Null_\tc R \simeq \Null_{\cc_S} R$, Dwyer's result \cite[3.2]{DwyerNoncommutative} shows that $R_\ff \cong S^{-1}R$. Finally, the quasi-isomorphism $\Null_\tc M \simeq \Null_\tc R \otimes_R^\mathbf{L} M$ implies $S^{-1}R\otimes_R M \cong M_\ff$.
\end{proof}
\end{example}

\begin{example}
This example is of a topological nature. Let $M$ be a discrete monoid and let $k=\Int/p\Int$ for some prime $p$. The ring $R$ we consider is the monoid ring $R=k[M]$, it has a natural augmentation $R \to k$ with augmentation ideal $I$. We also make the following assumptions:
\begin{enumerate}
\item The classifying space $\mathrm{B}M$ of $M$ has a finite fundamental group.
\item The augmentation ideal $I$ is finitely generated as a left $R$-module.
\item There is a projective resolution $P=\cdots P_2 \to P_1 \to P_0$ of $k$ over $R$ such that every $P_n$ is finitely generated as an $R$-module.
\end{enumerate}
Let $\tc$ be the hereditary torsion class of $I$-torsion $R$-modules, then $\langle\tc\rangle = \langle k \rangle$ and hence $\cell_\tc \simeq \cell_k$. Denote by $R^\vee$ the left $R$-module $\Hom_k(R,k)$. From the results of Dwyer, Greenlees and Iyengar \cite[6.15 and 7.5]{DwyerGreenleesIyengar} it is easy to conclude that $\cell_k R^\vee$ is quasi-isomorphic to the cochain complex (with coefficients in $k$) of a certain space we describe next. Let $(\mathrm{B}M)^\wedge_p$ be the Bousfield-Kan $p$-completion of the classifying space of $M$. The space $\Omega (\mathrm{B}M)^\wedge_p$ is the loop-space of $(\mathrm{B}M)^\wedge_p$. So, $\cell_k R^\vee$ is quasi-isomorphic to $\chains^*(\Omega (\mathrm{B}M)^\wedge_p;k)$ - the singular cochain complex of $\Omega (\mathrm{B}M)^\wedge_p$ with coefficients in $k$. By Theorem~\ref{the: Nullification as double endomorphism}, there exists an injective $R$-module $E$ such that
\[ H^n(\Omega (\mathrm{B}M)^\wedge_p;k) \cong \ext_{\End_R(E)}^{n-1}(\Hom_R(R^\vee,E),E) \text{ for } n\geq 2.\]
\end{example}

\bibliographystyle{amsplain}    
\bibliography{bib2007}          

\end{document}